\documentclass[oneside]{article}
\usepackage{ucs}
\usepackage[utf8x]{inputenc}
\usepackage[T1]{fontenc}
\usepackage{lmodern}
\usepackage{amsmath}
\usepackage{amsthm}
\usepackage{mathrsfs}
\usepackage{amsfonts}
\usepackage{array}
\usepackage{graphicx}
\usepackage{epstopdf}
\usepackage{bbm}
\usepackage{amssymb}
\usepackage{tikz}

\theoremstyle{definition} \newtheorem{definition}{Definition}[section]
\theoremstyle{definition} 
\theoremstyle{remark}     \newtheorem{remark}[definition]{Remark}
\theoremstyle{plain}      \newtheorem{theorem}[definition]{Theorem}
\theoremstyle{plain}      \newtheorem{proposition}[definition]{Proposition}
\theoremstyle{plain}      
\theoremstyle{plain}      \newtheorem{lemma}[definition]{Lemma}
\theoremstyle{definition} 
\theoremstyle{definition} \newtheorem{question}[definition]{Question}
\theoremstyle{plain}      \newtheorem{fact}{Fact}

  \newcommand{\norme}[1]
    {\left\| #1 \right\|}
  \newcommand{\abs}[1]
    {\left\vert #1 \right\vert}
  \newcommand{\ps}[2]
    {\left\langle #1,#2 \right\rangle}
  
  \newcommand{\ind}[1]
    {\mathbbm{1}_{#1}}
  \newcommand{\conj}[1]
    {\overline{#1}}
  \newcommand{\Range}
    {\mathrm{Ran}}
      
  \newcommand{\adherence}[1]{\overline{#1}}
  \newcommand{\Interieur}[1]{int(#1)}    
  
  \newcommand{\Z}{\mathbb{Z}}
  \newcommand{\B}[1]{\mathcal{B}(#1)}
  
  \newcommand{\N}{\mathbb{N}}
  
  \newcommand{\C}{\mathbb{C}}  
  
  \newcommand{\lun}{l^1}
  \newcommand{\lp}[1]{l^{#1}}  
  
  \newcommand{\Spec}[1]{\sigma(#1)} 
  \newcommand{\Specp}[1]{\sigma_p(#1)}
  \newcommand{\Spece}[1]{\sigma_e(#1)}
  
  \newcommand{\Span}[1]{\overline{span}\lbrace #1 \rbrace}
    
  \newcommand{\Reel}[1]{\mathsf{Re}(#1)}
  \newcommand{\Ima}[1]{\mathsf{Im}(#1)}
  
  \newcommand{\tens}[2]{#1 \otimes #2}
  \newcommand{\m}{m}

  \newcommand{\dist}{\textrm{dist}}  
  \newcommand{\diam}{\textrm{diam}}  
  \newcommand{\dens}{\textrm{dens}}

  \newcommand{\card}{\sharp}

\begin{document}

\title{Rank one perturbations of diagonal operators without eigenvalues}
\author{Hubert \sc{Klaja} 
       \thanks{D\'epartement de math\'ematiques et de statistique,
               Pavillon Alexandre-Vachon,
               1045, av. de la M\'edecine,
               Universit\'e Laval,
               Qu\'ebec (Qu\'ebec),
               Canada G1V 0A6 ; 
               \tt{hubert.klaja@gmail.com} 
               }
       }   
\date{ }

\maketitle

\begin{abstract} 
  In this paper, we prove that every diagonal operator on a Hilbert space of which is of multiplicity one and has perfect spectrum admits a rank one perturbation without eigenvalues. This answers a question of Ionascu. 
\end{abstract}

\section{Introduction}

  Let $H$ be an infinite dimensional separable complex Hilbert space, and let $(e_i)_{i \in \N}$ be a Hilbertian basis of $H$. If $u,v \in H$, we denote by $\tens{u}{v}$ the rank one operator defined for every $h \in H$ by 
$$
  \tens{u}{v}(h) = \ps{h}{v}u.
$$
  Recall that for every rank one operator $R \in \B{H}$, there exist $u,v \in H$ (not unique) such that $R = \tens{u}{v} $.
  We say that an operator $D$ on $H$ is  diagonal in the basis $(e_i)_{i \in \N}$ if there exists a sequence of complex numbers $(\lambda_i)_{i \in \N} $ such that
  $$
    D = \sum_{i \in \N} \lambda_i \tens{e_i}{e_i}.
  $$

  The class of operators which are rank one perturbations of a diagonal operator is still not well understood.
   For example, the invariant subspace problem as well as the hyperinvariant subspace problem, are still open for such operators (see \cite{Foias_Jung_Ko_Pearcy_2007}, \cite{Fang_Xia_2012} and \cite{Klaja_2013} for some partials results concerning the existence of hyperinvariant subspace for perturbations of diagonals operators). The most obvious approach for these operators is to look for an eigenvalue. But this is not always possible. Indeed Stampfli \cite{Stampfli_2001} built a diagonal operator $D$ and two vectors $u,v \in H$ such that $\Specp{D + \tens{u}{v}} = \emptyset$. We don't know if this operator has a non trivial invariant (nor hyperinvariant) subspace.
On the other hand, the opposite phenomenon can happen. Indeed it is proven in \cite{Sophie_2012} that there exists a rank one perturbation of a unitary diagonal operator which has uncountably many eigenvalues (see also \cite{Baranov_Lishanskii_2014} for an alternate proof).

  In \cite{Ionascu_2001}, Ionascu studied rank one perturbation of diagonal operators and asked the following question

\begin{question}[\cite{Ionascu_2001}]
  \label{QuestionIonascu}
  Let $D$ be a diagonal operator. Does there exist $u,v \in H$ such that 
  \begin{enumerate}
    \item $\Spec{D + \tens{u}{v}} = \Spec{D}$, 
    \item $\Specp{D + \tens{u}{v}} = \emptyset$ ?
  \end{enumerate}
\end{question}

  The goal of this article is to answer this question. Ionascu proved in \cite{Ionascu_2001} that if a bounded diagonal operator $D$ has no cyclic vector, or if the spectrum of $D$ has an isolated point, the answer is no.
The main result of this paper is the following result, which is a positive answer to Question \ref{QuestionIonascu} in all the other cases. 
 
\begin{theorem}
  \label{ThPertRangUnOpDiag}
  Let  $D = \sum_{i \in \N} \lambda_i \tens{e_i}{e_i} $ be a bounded diagonal operator on $H$. Suppose that $D$ has a cyclic vector and that $\Spec{D} $ is a perfect compact set (i.e. that $\Spec{D} $ has no isolated points). Then there exist $u,v \in H$ such that
  \begin{enumerate}
    \item $\Spec{D + \tens{u}{v}} = \Spec{D}$, 
    \item $\Specp{D + \tens{u}{v}} = \emptyset$.  
  \end{enumerate}
  Moreover $u,v \in H$ can be chosen so that $\norme{\tens{u}{v}}$ is arbitrarily small.
\end{theorem}
  
  Moreover it will be clear from the strategy of the proof that the results of \cite{Foias_Jung_Ko_Pearcy_2007} (and even those of \cite{Fang_Xia_2012} and \cite{Klaja_2013}) about the existence of a non trivial hyperinvariant subspace won't apply to the operators build in this theorem. Therefore for some of those operators, we won't know if they posses a non trivial hyperinvariant subspace.    
  
   The paper is organized  as follows, in section 2 we will recall some known results about rank one perturbations of diagonal operators that will help us to define a strategy to answer Question \ref{QuestionIonascu}. In section 3 we will recall the basic results needed to prove the main result.  In section 4 we build a vector that is not in the range of $D-z$ for any $z \in \Spec{D}$. More precisely we will prove the following proposition.
\begin{proposition}
  \label{PropUPasDansImDMoinsZIntro}
  Let $(\lambda_i)_{i \in \N}$ be a sequence of complex number dense in a closed compact perfect set $K$ and such that for every $i \ne j $, $\lambda_i \ne \lambda_j $.
  Then there exists a sequence $(u_i)_{i \in \N}$ of complex numbers such that 
  \begin{enumerate}
    \item for all $i \in \N$, $u_i \ne 0 $,
    \item for all $z \in K \setminus \lbrace \lambda_i, i \in \N \rbrace$,
    $$
    \sum_{i \in \N} \frac{\abs{u_i}^2}{\abs{z-\lambda_i}^2} = \infty.
    $$
  \end{enumerate}
\end{proposition}   
    In section 5 we will reproduce a proof due to William Alexandre \cite{William_Alexandre} (who kindly allowed the author to reproduce the proof here), for building an analytic function which does not vanish on a prescribed set.
\begin{proposition}
  \label{PropFctionHoloQuiSannulePas}
  Let $F \subset \C$ be a perfect closed set. Let $(\lambda_i)_{i \in \N}$ a dense sequence in $F$. Let $(\gamma_i)_{i \in \N}$ be a sequence of strictly positive numbers such that
  $$
    \sum_{i \in \N} \gamma_i < \infty.
  $$
  Then there exists a sequence of complex numbers $(c_i)_{i \in \N}$ such that
\begin{enumerate}
  \item for all $i \in \N$, $c_i \ne 0$,
  \item for all $i \in \N$, $\abs{c_i} \le \gamma_i $,
  \item $ \sum_{i=1}^\infty \frac{c_i}{z-\lambda_i}$ converges uniformly on every compact subset of $\C \setminus F $,
  \item for all $z \in \C \setminus F $, $ \sum_{i=1}^\infty \frac{c_i}{z-\lambda_i}- 1 \ne 0$. 
\end{enumerate}
\end{proposition}         
     In section 6 we give a proof of the main result and in section 7 we discuss a generalization of the main result to unbounded diagonal operators.

\section{Some results about rank one perturbation of diagonal operators}

In this section, we recall some results of Ionascu \cite{Ionascu_2001} concerning rank one perturbations of diagonal operators.

\begin{proposition}[\cite{Ionascu_2001}] 
  Let $D = \sum_{i \in \N} \lambda_i \tens{e_i}{e_i} $ be a diagonal operator. If there exists $i,j \in \N$ such that $i \ne j$ and $\lambda_i = \lambda_j $, then for all $u, v \in H $ we have that 
  $
  \lambda_i \in \Specp{D + \tens{u}{v}}. 
  $
\end{proposition}

  Recall that only diagonal operators of spectral multiplicity one possess cyclic vectors. So we can reformulate the previous proposition the following way:
if $D$ has no cyclic vector, the answer to Question \ref{QuestionIonascu} is no. Here is another result of Ionascu.

\begin{theorem}[Ionascu \cite{Ionascu_2001}] 
  Let $D = \sum_{i \in \N} \lambda_i \tens{e_i}{e_i} $ be a diagonal operator. Let $i \in \N$. If for all $j \ne i$, $\lambda_j \ne \lambda_i $, and if $\lambda_i $ is an isolated point of $\Spec{D}$, then for all $u,v \in H $, we have either
  \begin{enumerate}
    \item $\lambda_i \in \Specp{D + \tens{u}{v}} $,
    \item $\lambda_i \notin \Spec{D + \tens{u}{v}} $.
  \end{enumerate}
\end{theorem}

  This result tells us that if $\Spec{D}$ possesses some isolated points, then the answer to the Question \ref{QuestionIonascu} is no as well.

  Let $E \subset \C $ be a subset of the complex plane. We say that $E$ is a \emph{perfect} set if it has no isolated points.
  Summering the two results above we see that if $D$ has no cyclic vectors or if $\Spec{D}$ has an isolated point, then the answer to Question \ref{QuestionIonascu} is no.

The next result is again due to Ionascu, and gives necessary and sufficient conditions for $z$ to be an eigenvalue of $D + \tens{u}{v} $.

\begin{proposition}[Ionascu \cite{Ionascu_2001}] 
  \label{PropIonascuCnsRangUnVp}
  Let $D = \sum_{i \in \N} \lambda_i \tens{e_i}{e_i} $ be a diagonal operator such that for every $i \ne j $, we have $\lambda_i \ne \lambda_j $. Let $u,v \in H $ be two vectors such that for every $i \in \N$ we have $\ps{u}{e_i} \ne 0 $ and $\ps{v}{e_i} \ne 0 $. Then $z \in \Specp{D + \tens{u}{v}}$ if and only if
  \begin{enumerate}
    \item $z \notin \Specp{D}$,
    \item $\sum_{i \in \N} \frac{\abs{\ps{u}{e_i}}^2}{\abs{z - \lambda_i}^2} < \infty $,
    \item $\sum_{i \in \N} \frac{\ps{u}{e_i} \conj{\ps{v}{e_i}}}{z-\lambda_i} = 1 $.
  \end{enumerate}
\end{proposition}

  The condition (1) from Proposition \ref{PropIonascuCnsRangUnVp} states that if $z=\lambda_i$, then $z$ cannot be an eigenvalue of $D + \tens{u}{v} $. We remark that condition (2) is equivalent to the fact that $u$ belongs to $ \Range(D-z)$. 
  Remark that if $z \notin \Spec{D}$, as $D-z$ is invertible, we have that condition (2) is automatically satisfied. 
  
  In order to prove Theorem \ref{ThPertRangUnOpDiag}, we will build a vector $u \in H$ such that for all $i \in \N$, $\ps{u}{e_i} \ne 0 $ and for all $z \in \Spec{D} \setminus \Specp{D}$ we have
  $$
    \sum_{i \in \N} \frac{\abs{\ps{u}{e_i}}^2}{\abs{z - \lambda_i}^2} = \infty. 
  $$
  In this way, condition (2) of Proposition \ref{PropIonascuCnsRangUnVp} will not be satisfied when $ z \in \Spec{D} \setminus \Specp{D} $, and this will prove that $(\Spec{D} \setminus \Specp{D}) \cap \Specp{D + \tens{u}{v}} = \emptyset $.
  
  Then we will construct a vector $v \in H$ such that for all $i \in \N$, $\ps{v}{e_i} \ne 0 $ and for all $z \in \C \setminus \Spec{D}$ 
  $$
    \sum_{i \in \N} \frac{\ps{u}{e_i} \conj{\ps{v}{e_i}}}{z-\lambda_i} \ne 1.
  $$
  This boils down to construct an analytic function of the form
  $$
    \sum_{i \in \N} \frac{c_i}{z- \lambda_i} - 1
  $$
  which does not vanish on $\C \setminus \Spec{D} $, with a sumability condition on the coefficients $c_i$ that will guaranty that for a suitable choice of $u$ and $v \in H$, $c_i = \ps{u}{e_i} \conj{\ps{v}{e_i}}$. So condition (3) of Proposition \ref{PropIonascuCnsRangUnVp} won't be satisfied for all $z \in \C \setminus \Spec{D}$. So we will have that $(\C \setminus \Spec{D}) \cap \Specp{D + \tens{u}{v}} = \emptyset $. According to condition (1), we will have that $\Specp{D} \cap \Specp{D + \tens{u}{v}} = \emptyset $ and $D + \tens{u}{v} $ won't have any eigenvalue. Thus a positive answer to Question \ref{QuestionIonascu} will follow.
  
\section{Preliminaries}

  Before carrying out the two steps in the proof of Theorem \ref{ThPertRangUnOpDiag}, we will need tools, which we present in this section. The first result is a classical theorem from measure theory. A more general version of this one can be found in in \cite[Th~32]{Rogers_1970}.
  
\begin{theorem}
  \label{ThEnsMesNulleRogers}
  Let $E \subset \C $ be a measurable subset of the complex plane of Lebesgue measure zero. Then there exists a family of open balls $(O_i)_{i \in \N} $ such that
    $ E \subset \cup_{i \in \N} O_i $,
     $\sum_{i \in \N} \diam(O_i)^2 < \infty $,
    and for all $z \in E$, there exist infinitely many $i \in \N$ such that $z \in O_i$.
\end{theorem}  

  We will also use the notion of point of Lebesgue density of points of a measurable subset of the complex plane. 
  In the following,  $B$ will always denote a ball of the complex plane.
  Recall that if $E \subset \C$ is a measurable subset of the complex plane, and $z \in \C$, we say that $z$ is a point of Lebesgue density of $E$ if
  $$
    \dens(z,E) = \lim_{\m(B)\rightarrow 0, z \in B} \frac{\m(E \cap B)}{\m(B)} = 1.
  $$ 
  
  If $z$ is a point of Lebesgue density of $E$, then there are "a lot" of points belonging to $E$ around $z$ (in the sense of the Lebesgue measure). 
 Recall that Lebesgue density theorem says that for almost every $z \in \C$,
  $
    \dens(z,E) = \ind{E}(z)
  $.
  It is also possible to replace the balls $B$ by a family a measurable subset that shrinks regularly to $z$. 
 Remind that a collection of measurable subsets $(U_\alpha)_{\alpha \in A}$ of $\C$ is said to shrink regularly to $z$ if there exists a constant $c>0$ such that for all $U_\alpha$, there exists a ball $B$ such that 
  $
    z \in B, \, U_\alpha \subset B \, \textrm{ and } \, \m(U_\alpha) \ge c \m(B)
  $.
  More information about Lebesgue density can be found in \cite{Stein_Shakarchi_2005}.

\section{A vector not in $\Range(D-z)$}

The goal of this section is to prove that if $D$ satisfy the hypothesis of Theorem \ref{ThPertRangUnOpDiag}, 
then there exists a vector $u \in H$ such that for every $z \in \Spec{D} \setminus \Specp{D} $, $u$ does not belong to $\Range{(D - z)}$. First we rephrase Proposition \ref{PropUPasDansImDMoinsZIntro}.

\begin{proposition}
  \label{PropUPasDansImDMoinsZ}
  Let $D = \sum_{i \in \N} \lambda_i \tens{e_i}{e_i}$ be a bounded diagonal operator such that for every $i \ne j $, $\lambda_i \ne \lambda_j $ and $\Spec{D} $ is a perfect compact set.
  Then there exists a vector $u \in H$ such that 
  \begin{enumerate}
    \item for all $i \in \N$, $\ps{u}{e_i} \ne 0 $,
    \item for all $z \in \Spec{D} \setminus \Specp{D}$,
    $$
    \sum_{i \in \N} \frac{\abs{\ps{u}{e_i}}^2}{\abs{z-\lambda_i}^2} = \infty.
    $$
  \end{enumerate}
\end{proposition}

The proof of this Proposition is inspired from a paper of Stampfli \cite{Stampfli_2001}. We will divide the proof in several lemmas.

  Without loss of generality, we can suppose that $\Spec{D} \subset [0,1]\times[0,1]$. For all $n \in \N$, we denote by $(C_{n,k})_{k=0}^{2^{2n}}$ the family of closed dyadic squares $[i2^{-n},(i+1)2^{-n}]\times[j2^{-n},(j+1)2^{-n}] $ with $i,j \in \lbrace 0, \dots, 2^n - 1 \rbrace $.
 We denote by $D_{n,k} $ the interior of the square $C_{n,k}$. We also denote by $F_n$ the boundary of all the dyadic squares at stage $n$, i.e.
  $$
  F_n = \bigcup_{k=0}^{2^n} \lbrace (x,y) \in [0,1]\times[0,1]: x= k2^{-n} \textrm{ or } y = k2^{-n} \rbrace.
  $$


  We set
  \begin{align*}
    A_1 = \lbrace z \in \Spec{D}: \dens(z,\Spec{D}) = 1 \rbrace \setminus \cup_{n \in \N} F_n 
    \quad \textrm{and} \quad
    A_2 = \Spec{D} \setminus A_1.
  \end{align*}
  
  Using Lebesgue density Theorem 
and the fact that $\m(\cup_{n} F_n) = 0 $, we get that $\m(A_1) = \m(\Spec{D})$ and $\m(A_2) =0 $.  
  
\begin{lemma}
  If $D_{n,k} \cap A_1 \ne \emptyset $, then there are infinitely many eigenvalues of $D$ lying inside $D_{n,k}$.
\end{lemma}  
  
\begin{proof}
  Let $z \in D_{n,k} \cap A_1$. From the definition of $A_1$, we have $ \dens(z,\Spec{D}) = 1 $. In other words, if we denote by $B(z,r) $ the ball centered at $z$ of radius $r$, we get that
  $$
  \lim_{r \rightarrow 0} \frac{\m(B(z,r)\cap \Spec{D})}{\m(B(z,r))} = 1.
  $$
  In order to prove the Lemma, we will prove by contradiction that $\m(D_{n,k} \cap \Spec{D}) > 0 $, as $\Spec{D}$ is a perfect set and $\Specp{D}$ is dense in $\Spec{D} $, this will prove the lemma. Suppose that $\m(D_{n,k} \cap \Spec{D}) = 0 $. Then
  \begin{align*}
  \m(B(z,r) &\cap \Spec{D}) \\
    &= \m(B(z,r) \cap D_{n,k} \cap \Spec{D}) + \m(B(z,r) \cap (\C \setminus D_{n,k}) \cap \Spec{D}) \\
    &= \m(B(z,r) \cap (\C \setminus D_{n,k}) \cap \Spec{D}).
  \end{align*}
  When $r$ is small enough, we get that $B(z,r) \subset D_{n,k}$ (because $D_{n,k}$ is open), so $\m(B(z,r) \cap (\C \setminus D_{n,k}) \cap \Spec{D}) = \m(\emptyset) = 0$. So $\m(B(z,r) \cap \Spec{D})= 0$ and 
    $$
  \lim_{r \rightarrow 0} \frac{\m(B(z,r)\cap \Spec{D})}{\m(B(z,r))} = 0.
  $$
  This contradicts the assumption that $z \in A_1$ and $ \dens(z,\Spec{D})=1 $.
\end{proof}  
  
  Since $\m(A_2)=0$, Theorem \ref{ThEnsMesNulleRogers} implies that there exists a family $(O_i)_{i \in \N} $ of open subsets of the complex plane $\C$ such that $ A_2 \subset \cup_{i \in \N} O_i $, for all $i \in \N$, $O_i \cap \Spec{D} \ne \emptyset $, $\sum_{i \in \N} \diam(O_i)^2 < \infty $ and for all $z \in A_2$, there exist infinitely many $i \in \N$ such that $z \in O_i$. Now we can renumber the eigenvalues $\lambda_i$ and the eigenvectors $e_i$ of $D$ by stage.  

We initialize at stage $0$ ($n = 0$). We have that $C_{0,1} = [0,1]\times[0,1]$ and $D_{0,1} = ]0,1[\times]0,1[$. If $D_{0,1} \cap A_1 \ne \emptyset $, we choose $i(0,1) \in \N$ such that $\lambda_{i(0,1)} \in D_{0,1} \cap \Spec{D} $. Otherwise we do nothing. Write $E_0 = \lbrace k \in \N : D_{0,k} \cap A_1 \ne \emptyset \rbrace$, and  $I_0 = \lbrace i(0,k):  k \in E_0 \rbrace$.

Then we choose $j(0) \in \N \setminus I_0$ such that $\lambda_{j(0)} \in \Spec{D} \cap O_0 $ (we can do it because $O_0 \cap \Spec{D} \ne \emptyset $, $\Spec{D}$ is a perfect set and $O_0 $ is open). Denote by $J_0 = \lbrace j(0) \rbrace$.

Once the stages $0, \dots, n-1$ are complete, we proceed with stage $n$ 
. For each $k \in \lbrace 1, \dots, 2^{2n} \rbrace $, if $D_{n,k} \cap A_1 \ne \emptyset $, we choose $i(n,k) \in \N \setminus (J_{n-1} \cup (\cup_{k=0}^{n-1} I_k))$ such that $\lambda_{i(n,k)} \in D_{n,k} \cap \Spec{D} $.  Otherwise we do nothing. Denote by $E_n = \lbrace k \in \N : D_{n,k} \cap A_1 \ne \emptyset \rbrace$, and $I_n = \lbrace i(n,k): k \in E_n \rbrace$.

We choose $j(n) \in \N \setminus (J_{n-1} \cup (\cup_{k=0}^n I_k))$ such that $\lambda_{j(n)} \in \Spec{D} \cap O_n $ (we can do it because $O_n \cap \Spec{D} \ne \emptyset $, $\Spec{D}$ is a perfect set and $O_n$ is open). Denote by $J_n = \lbrace j(k): k=0, \dots, n \rbrace$.

From the construction, we get for all $n,m \in \N$ such that $n \ne m $ that
$$
 J_n \cap J_m = \emptyset, \quad I_n \cap I_m = \emptyset, \quad J_n \cap I_m = \emptyset, \quad J_n \cap I_n = \emptyset.
$$
In other words, the $\lambda_{i(n,k)}$ and the $\lambda_{j(n)} $ are all distinct.


At the end of the renumbering, it is possible that we "forgot" some $\lambda_i$ (i.e. it is possible that $\N \ne (\cup_{n \in \N} I_n) \cup (\cup_{m \in \N} J_m)$). We will decompose our Hilbert space $H$ into three subspaces: $H = H_1 \oplus H_2 \oplus H_{reste} $ with
\begin{align*}
  H_1 &= \Span{ e_{i(n,k)}: n \in \N,  k  \in E_n} = \Span{ e_i : i \in \cup_{n \in \N} I_n }, \\
  H_2 &= \Span{ e_{j(n)}: n \in \N} = \Span{ e_j : j \in \cup_{n \in \N} J_n }, \\
  H_{r} &= \Span{ e_i: i \notin \cup_{n \in \N} I_n \cup_{m \in \N} J_m}.
\end{align*}

 Denote by $\beta_n$ the number of elements in $I_n$, i.e. the number of eigenvalues chosen in the dyadic square at stage $n$. Define the vectors $u_1, u_2$ and $u_{r}$ in the following way:
 $$ 
 u_1 = \sum_{n \in \N} \sum_{k \in E_n} \frac{1}{n\sqrt{\beta_n}} e_{i(n,k)}, \quad
 u_2 = \sum_{n \in \N} \diam(O_n) e_{j(n)}
 $$
 $$ \quad \textrm{and} \quad
 u_{r} = \sum_{i \in \N: \forall n \in \N, i \notin I_n \cup J_n} \frac{1}{i} e_i .
 $$

We have that
$$
  \norme{u_1}^2
    = \sum_{n \in \N} \sum_{k \in E_n} \abs{\ps{u}{e_{i(n,k)}}}^2 
    = \sum_{n \in \N} \sum_{k \in E_n} \frac{1}{n^2 \beta_n} 
    = \sum_{n \in \N} \frac{1}{n^2} < \infty. 
$$
So we have that $u_1 \in H_1$.
Concerning $u_2$, we have
$$
  \norme{u_2}^2 = \sum_{n \in \N} \abs{\ps{u}{e_{j(n)}}}^2 = \sum_{n \in \N} \diam(O_n)^2 < \infty.
$$
So we have $u_2 \in H_2$. We also get that $u_{r} \in H_{r}$. Denote by $u = u_1 + u_2 + u_{r}$. We have that $u \in H$.

In order to complete the proof of Proposition \ref{PropUPasDansImDMoinsZ} we will show that if $z \in A_1 \setminus \Specp{D} $ then $u_1 \notin \Range(D-z) $, and if $z \in A_2 \setminus \Specp{D}$, then $u_2 \notin \Range(D-z)$. This will enable us to conclude the proof of Proposition \ref{PropUPasDansImDMoinsZ}. Indeed if $z \in \Spec{D}$ then either $z \in A_1 $, or $z \in A_2 $. If $z \in A_i $, $i =1,2$, then $u_i \notin \Range(D-z) $. So $u \notin \Range(D-z) $.

\begin{lemma}
  Let $z \in A_1 \setminus \Specp{D} $. Then $u_1 \notin \Range(D-z) $.
\end{lemma}


\begin{proof}
  Let $z \in A_1 \setminus \Specp{D} $. For all $n \in \N$, there exists some unique integers $l_n(z)$ and $ m_n(z) \in \lbrace 0, \dots , 2^n \rbrace$, such that 
  $$
  z \in ]l_n(z) 2^{-n},(l_n(z)+1) 2^{-n}[\times ]m_n(z) 2^{-n},(m_n(z)+1) 2^{-n}[. 
  $$
   For all $p \in \lbrace 1, \dots, 2^n \rbrace $, set
 \begin{align*}
 L_{n,p}(z) 
 =& \quad \quad \quad \bigcup & [l 2^{-n},(l+1) 2^{-n}]\times[m 2^{-n},(m+1) 2^{-n}]. \\
  & m,l \in \lbrace 0, \dots , 2^{2n} \rbrace & \\
  & \abs{m - m_n(z)} \le p & \\
  & \abs{l - l_n(z)} \le p &
  \end{align*}
 
 The set $L_{n,p}(z) $ is the union of the closed squares $C_{n,k}$ which are at most $p$ squares away from $z$.
\begin{fact}   
    The family $(L_{n,p}(z))_{n \in \N, p \in \lbrace 1, \dots, 2^n \rbrace} $ shrinks regularly to $z$. 
\end{fact}

\begin{proof}[Proof of Fact 1]
  Indeed, if $\Reel{z} < 2^{-1}$ and $\Ima{z} < 2^{-1} $ (the cases $\Reel{z} > 2^{-1}$ and $\Ima{z} < 2^{-1} $, $\Reel{z} < 2^{-1}$ and $\Ima{z} > 2^{-1} $, $\Reel{z} > 2^{-1}$ and $\Ima{z} > 2^{-1} $ are similar),  $ L_{n,p}(z) $ contains the following smaller square
   $$
   P_{n,p} = [l_n(z) 2^{-n},(l_n(z)+p) 2^{-n}]\times[m_n(z) 2^{-n},(m_n(z)+p) 2^{-n}].
   $$
We have that $\m(P_{n,p}) = \frac{(p+1)^2}{2^{2n}} $. Moreover $ L_{n,p}(z) $ is a subset of the following bigger square
   $$
   S_{n,p} = [(l_n(z)-p) 2^{-n},(l_n(z)+p) 2^{-n}]\times[(m_n(z)-p) 2^{-n},(m_n(z)+p) 2^{-n}].
   $$   
   As $S_{n,p}$ is a square, there exists a ball $B_{n,p}$ such that $S_{n,p} \subset B_{n,p} $ and 
  $$
  \m(B_{n,p}) = \frac{\pi}{2} \m(S_{n,p}) = \frac{\pi}{2} \frac{(2p+1)^2}{2^{2n}} \le \frac{\pi}{2} 4\frac{(p+1)^2}{2^{2n}} = 2 \pi \m(P_{n,p}).
  $$
  In other words we have that $L_{n,p}(z) \subset B_{n,p} $ and 
  $$
  \frac{1}{2\pi}\m(B_{n,p}) \le \m(P_{n,p}) \le \m(L_{n,p}(z)).
  $$
  So $L_{n,p}(z) $ shrinks regularly to $z$.
\end{proof}    


  As $\dens(z,\Spec{D})=1 $, it follows 
that there exists $\varepsilon >0$ such that for all $L_{n,p}(z)$ such that $\m(L_{n,p}(z))< \varepsilon$ we have that
\begin{equation}
  \label{EqBcpDePts}
  \frac{\m(\Spec{D} \cap L_{n,p}(z))}{\m(L_{n,p}(z))} > \frac{3}{4}.
\end{equation}

Let $n \in \N$. If $ m_n(z)-p+1  \ge 0$, $ l_n(z)-p +1 \ge 0$, $ m_n(z) +p  \le 2^n$ and $ l_n(z) +p  \le 2^n$, then $L_{n,p}(z)$ is
$$
  [(l_n(z)-p+1) 2^{-n},(l_n(z)+p) 2^{-n}]\times[(m_n(z)-p+1) 2^{-n},(m_n(z) +p) 2^{-n}]
$$
  is a square. Denote by $G_n(z)$ the set of all integers $p \in \lbrace 1, \dots, 2^n \rbrace$ satisfying this condition. In other words, if we write
  $$
  p_n(z) = \min \lbrace m_n(z)+1, l_n(z) +1, 2^n -m_n(z), 2^n - l_n(z) \rbrace,
  $$
  we have that
  $$
  G_n(z) = \lbrace p \in \N: 1 \le p \le p_n(z) \rbrace.
  $$

  Let $\alpha \in \N $ be the smallest integer such that $\varepsilon \ge 4 \, 2^{-2\alpha}$. Let $p \in \N$ be such that $1 \le p \le 2^{n-\alpha}$. We have that 
$$
  \m(L_{n,p}(z))
    \le (2p-1)^2 2^{-2n}  
    \le  4p^2 2^{-2n}  
    \le  4 \, 2^{2n-2\alpha} 2^{-2n}  
    = 4\, 2^{-2\alpha} 
    \le \varepsilon.
$$
  So if $1 \le p \le 2^{n-\alpha} $, then $L_{n,p}(z)$ satisfies (\ref{EqBcpDePts}).
  
  We have that $p_n(z) \sim 2^n \min \lbrace \Reel{z}, \Ima{z}, 1- \Reel{z}, 1- \Ima{z} \rbrace $. We fix some $\varepsilon > 0 $ such that (\ref{EqBcpDePts}) is satisfied and 
  $$
    \frac{\sqrt{\varepsilon}}{2} \le \inf_{n \in \N} p_n(z) 2^{-n}.
  $$
  Then we have that
  $$
    2^{n-\alpha} \le  \frac{\sqrt{\varepsilon}}{2} 2^n \le \inf_{n \in \N} p_n(z) \le p_n(z).  
  $$
  In other words, if $\varepsilon$ is small enough, the sets $L_{n,p}(z)$ for $1 \le p \le 2^{n-\alpha} $ are squares, composed by $(2p-1)^2 $ squares $C_{n,k}$. From now on, we suppose that this condition is satisfied.
  
  Fix $n \ge 1$. We will prove by induction on $p \in \lbrace 1, \dots, 2^{n-\alpha} \rbrace$ that there exists a subset $I'_{n,p}$ of $I_n$ of cardinal $ p$, such that the $I'_{n,p}$, $p \in \lbrace 1, \dots, 2^{n-\alpha} \rbrace$, are pairwise disjoint, and for all $i \in I'_{n,p}$, $\lambda_i \in L_{n,p}(z)$.
  
Let $ n \ge n_0$ such that $1 \ge 2^{n_0 - \alpha} $. If $p = 1$, then $L_{n,1}(z)$ is the only square $ D_{n,k} $ which contains $z$. As $z \in A_1$ and $ z \in D_{n,k} $ then $ A_1 \cap D_{n,k} \ne \emptyset $. So $k \in E_n$ and $\lambda_{i(n,k)} \in D_{n,k} $, and we set $I'_{n,1} = \lbrace i(n,k) \rbrace $.

Inside $L_{n,p}(z) $, there are $(2p-1)^2$ squares $D_{n,k}$. As (\ref{EqBcpDePts}) is satisfied, we get that
$$
  \m(L_{n,p}(z) \cap \Spec{D}) = \m(L_{n,p}(z) \cap A_1) > \frac{3}{4} \m(L_{n,p}(z)).
$$
There are at least $\frac{3}{4}$ of the squares $D_{n,k}$ forming $L_{n,p}(z) $ which meet $A_1$. Otherwise there would be a least $\frac{1}{4}$ of the $(2p-1)^2$ squares $D_{n,k}$ included in $L_{n,p}(z) $ which does not meet $ A_1$, so we would have
 $$
   \m(L_{n,p}(z) \setminus \Spec{D}) \ge \frac{1}{4} (2p-1)^2 2^{-2n}.
 $$
 Consequently
 \begin{align*}
   \m(L_{n,p}(z) \cap \Spec{D}) 
     &= \m(L_{n,p}(z)) - \m(L_{n,p}(z) \setminus \Spec{D}) \\ 
     &\le (2p-1)^2 2^{-2n} - \frac{1}{4} (2p-1)^2 2^{-2n} \\
     &= \frac{3}{4} (2p-1)^2 2^{-2n} \\
     &= \frac{3}{4} \m(L_{n,p}(z)).  
 \end{align*}
 This contradicts (\ref{EqBcpDePts}). 
 
 In other words, we chose during the construction $u$ at least $\frac{3(2p-1)^2}{4}$ squares at stage $n$ in $ L_{n,p}(z)$. Denote by $J'_{n,p}$ the set of the  corresponding index $i(n,k)$, i.e. $J'_{n,p} = \lbrace i(n,k) : D_{n,k} \cap A_1 \ne \emptyset \textrm{ and } D_{n,k} \subset L_{n,p}(z) \rbrace $. We have that $J'_{n,p} \subset I_n$.  Let $I''_{n,p} = J'_{n,p} \setminus \cup_{l=1}^{p-1} I'_{n,l} $. We have that for all $i \in I''_{n,p}$, $\lambda_i \in L_{n,p}(z)$ so 
 $$
  \abs{z - \lambda_i} < \frac{p\sqrt{2}}{2^n}.   
$$ 
Moreover, since the cardinal of $I''_{n,l}$ is $l$ for each $l \in \lbrace 1, \dots, p-1 \rbrace $
\begin{align*}
  \card(I''_{n,p}) 
    \ge \frac{3}{4}(2p-1)^2 - \sum_{l=1}^{p-1} l 
    \ge \frac{3}{4}(2p-1)^2 - \frac{p(p-1)}{2} 
    \ge p.
\end{align*}
  So we can choose for $I'_{n,p}$ any subset of $I''_{n,p}$ of cardinal $p$. 
  
  Then the $I'_{n,p}$ are pairwise disjoint, and contained in $I_n$. We have that
\begin{align*}
  \sum_{k \in E_n} \frac{1}{\abs{z- \lambda_{i(n,k)}}^2}
    &= \sum_{i \in I_n} \frac{1}{\abs{z- \lambda_{i}}^2} 
    \ge \sum_{p=1}^{2^{n-\alpha}} \sum_{i \in I'_{n,p}} \frac{1}{\abs{z- \lambda_{i}}^2} 
    \ge \sum_{p=1}^{2^{n-\alpha}} \frac{2^{2n}}{2p^2}p \\ 
    &=   \frac{2^{2n}}{2} \sum_{p=1}^{2^{n-\alpha}} \frac{1}{p} 
    \ge  \frac{2^{2n}}{2} \log(2^{n-\alpha}) 
    =    \frac{2^{2n}}{2} (n-\alpha) \log(2).
\end{align*}
  So
\begin{align*}
  \norme{(D-z)^{-1}u_1}^2 
    &\ge \sum_{n \in \N} \sum_{k\in E_n} \frac{1}{n^2\beta_n} \frac{1}{\abs{z-\lambda_{i(n,k)}}^2} \\ 
    &=   \sum_{n \in \N}  \frac{1}{n^2\beta_n} \sum_{k\in E_n} \frac{1}{\abs{z-\lambda_{i(n,k)}}^2} \\ 
    &\ge \sum_{n \in \N}  \frac{1}{n^2\beta_n} \frac{2^{2n}}{2} (n-\alpha) \log(2) \\
    &=   \frac{\log(2)}{2} \sum_{n \in \N}  (n-\alpha) \frac{2^{2n}}{n^2\beta_n} \\
    &\ge  \frac{\log(2)}{2} \sum_{n \in \N} \frac{n-\alpha}{n^2} \\
    &=  \infty.
\end{align*}
  We have used here the fact that there are at most $2^{2n}$ dyadic squares $C_{n,k}$ at stage $n$, and consequently $\beta_n$ cannot exceed $2^{2n} $. So $\frac{2^{2n}}{\beta_n} \ge 1 $. This proves that if $z \in A_1 \setminus \Specp{D}$ then $u_1 \notin \Range(D-z)$.
\end{proof}

\begin{lemma}
  Let $z \in A_2 \setminus \Specp{D} $. Then $u_2 \notin \Range(D-z) $.
\end{lemma}

\begin{proof}
  Let $z \in A_2 \setminus \Specp{D}$. Set $J_z = \lbrace i \in \N : z \in O_i \rbrace$. Then we have that
\begin{align*}
  \norme{(D-z)^{-1}u_2}^2
    &=   \sum_{n \in \N} \frac{\abs{\ps{u}{e_{j(n)}}}^2}{\abs{z-\lambda_{j(n)}}^2} 
    \ge \sum_{i \in J_z} \frac{\abs{\ps{u}{e_{j(i)}}}^2}{\abs{z-\lambda_{j(i)}}^2} 
    =   \sum_{i \in J_z} \frac{\diam(O_i)^2}{\abs{z-\lambda_{j(i)}}^2} \\ 
    &\ge \sum_{i \in J_z} \frac{\diam(O_i)^2}{\diam(O_i)^2} 
    =   \sum_{i \in J_z} 1 
    =   \infty,
\end{align*}
because there are infinitely many $i$ such that $z \in O_i $. This prove that if $z \in A_2 \setminus \Specp{D} $, then $u_2 \notin \Range(D-z) $.
\end{proof}

We conclude that if $z \in \Spec{D}\setminus \Specp{D}$, then $u \notin \Range(D-z)  $. This finishes the proof of Proposition \ref{PropUPasDansImDMoinsZ}.

\section{An analytic function which does not vanish outside a perfect set}

  Thanks to the previous work, we can make condition (2) of Proposition \ref{PropIonascuCnsRangUnVp} impossible to be satisfied for any $z \in \Spec{D} \setminus \Specp{D}$. In other words, we can build a rank one perturbation of $D$ without eigenvalues inside $\Spec{D}$. In this section we present a tool that will allow us to construct a rank one perturbation of $D$ without any eigenvalue outside  $\Spec{D}$. In order to do so, we need to make sure that condition (3) of Proposition \ref{PropIonascuCnsRangUnVp} is satisfied for any $z \in \C \setminus \Spec{D}$. The results of this section are due to William Alexandre \cite{William_Alexandre}, who kindly allowed the author to reproduce it here.
  
\begin{proof}[Proof of Proposition \ref{PropFctionHoloQuiSannulePas}]
  We want to construct a function $f$ analytic on $\C \setminus F$ of the form
  $$
    f(z) = \sum_{i=1}^\infty \frac{c_i}{z-\lambda_i} - 1
  $$
  which does not vanish on $\C \setminus F$. If we choose $f$ of the form 
  $$
  f(z)=  \prod_{i=1}^\infty \frac{z- \mu_i}{z - \lambda_i},
  $$
  with $\mu_i \in F \setminus \lbrace \lambda_i: i \in \N \rbrace$ well chosen, such that the infinite product converges uniformly on every compact subset of $\C \setminus F$, this will allow us to prove that $f$ does not vanish on $ \C \setminus F$.
  The $\mu_i$ will be constructed by induction. At each stage $N$, we will consider the partial product
  $$
  f_N(z) = \prod_{i=1}^N \frac{z- \mu_i}{z - \lambda_i},
  $$
  and prove that $f_N$ can be written as
  $$
  \sum_{i=1}^N \frac{c_{i,N}}{z-\lambda_i} - 1.
  $$
  With a suitable choice of $\mu_i$, we will prove that we can control the $c_{i,N}$. These will converge to some $c_{i,\infty} = c_i$ as $N$ tends to infinity, and thus will give a natural candidate for $f$ of the form
  $$
  f(z) = \sum_{i=1}^\infty \frac{c_i}{z-\lambda_i} - 1.
  $$
  We will check then that $f(z) \ne 0$ for every $z \in \C \setminus F$. This will be a consequence of the fact that $f_N$ will converge to $f$ uniformly on every compact subset of $\C \setminus F$. 

  Let $c_{1,N}, \dots, c_{N,N} \in \C $ be some complex numbers. We have that 
$$
  \sum_{i=1}^N \frac{c_{i,N}}{z-\lambda_i} - 1 = \frac{\sum_{i=1}^N c_{i,N} \prod_{j=1,\, j \ne i}^N (z - \lambda_j) - \prod_{i=1}^N (z- \lambda_i) }{\prod_{i=1}^N (z- \lambda_i) }.
$$
If we want that
$$
  \sum_{i=1}^N \frac{c_{i,N}}{z-\lambda_i} - 1 = \prod_{i=1}^N \frac{z- \mu_i}{z - \lambda_i},
$$ 
we must have that
$$
  \sum_{i=1}^N c_{i,N} \prod_{j=1,\, j \ne i}^N (z - \lambda_j) - \prod_{i=1}^N (z- \lambda_i)  = \prod_{i=1}^N (z- \mu_i).
$$
Evaluated at point $z= \lambda_k$, this last inequality can be rewritten as
$$
  c_{k,N}\prod_{j=1, \, j \ne k}^N (\lambda_k - \lambda_j)= \prod_{i=1}^N (\lambda_k - \mu_i).
$$ 
If we denote
$$
  c_{k,N}= (\lambda_k - \mu_k) \prod_{j=1, \, j \ne k}^N \frac{\lambda_k-\mu_j}{\lambda_k - \lambda_j},
$$
we have that
$$
  f_N(z) = \prod_{i=1}^N \frac{z- \mu_i}{z - \lambda_i} =  \sum_{i=1}^N \frac{c_{i,N}}{z-\lambda_i} - 1.
$$
Let $(\epsilon_i)_{i \in \N}$ be a sequence of positive real numbers such that
$$
  \prod_{i=1}^\infty (1+\epsilon_i) < \infty.
$$
Now, as $F$ is a perfect set, we can choose by induction the $\mu_k \in F \setminus \lbrace \lambda_i: i \in \N \rbrace$, such that for every $j<k$ we have that
$$
  \frac{\abs{\lambda_k - \mu_k}}{\abs{\lambda_j-\mu_k}} < \epsilon_k,
\quad \textrm{and} \quad
  \abs{(\lambda_k - \mu_k)\prod_{i=1}^{k+1} \frac{\lambda_k - \mu_i}{\lambda_k - \lambda_j}} < \frac{\gamma_k}{\prod_{i=k+1}^\infty (1+\epsilon_i)}.
$$
Denote by $c_{k,N}$ the coefficient associated to those $\mu_k$. Namely we set
$$
  c_{k,N}= (\lambda_k - \mu_k) \prod_{j=1, \, j \ne k}^N \frac{\lambda_k-\mu_j}{\lambda_k - \lambda_j}.
$$
We then denote by
$$
  c_k = c_{k,\infty} = \prod_{j=1}^{k-1} \frac{\lambda_k-\mu_j}{\lambda_k - \lambda_j}(\lambda_k - \mu_k) \prod_{j=k}^\infty \frac{\lambda_k-\mu_j}{\lambda_k - \lambda_j}.
$$
We have that
\begin{align*}
  \abs{c_k}
    &= \abs{\prod_{j=1}^{k-1} \frac{\lambda_k-\mu_j}{\lambda_k - \lambda_j}(\lambda_k - \mu_k)} \abs{\prod_{j=k+1}^\infty \frac{\lambda_k-\mu_j}{\lambda_k - \lambda_j}} \\
    &= \abs{\prod_{j=1}^{k-1} \frac{\lambda_k-\mu_j}{\lambda_k - \lambda_j}(\lambda_k - \mu_k)} \abs{\prod_{j=k+1}^\infty\left( 1 + \frac{\lambda_j-\mu_j}{\lambda_k - \lambda_j}\right) }  \\
    &\le \frac{\gamma_k}{\prod_{i=k+1}^\infty (1+\epsilon_i)} \prod_{i=k+1}^\infty (1+\epsilon_i) \\
    &= \gamma_k.
\end{align*}
  As $(\gamma_k)_{k \in \N} \in \lun$, we have that $(c_k)_{k \in \N} \in \lun$. As $c_k$ can be written as a convergent product of complex numbers, i.e.
  $$  
   c_k = K_k \prod_{j=k+1}^\infty (1 + \alpha_j)
  $$
with
$$
  K_k = \prod_{j=1}^{k-1} \frac{\lambda_k-\mu_j}{\lambda_k - \lambda_j}(\lambda_k - \mu_k),
$$ 
and
$$
  \alpha_j = \frac{\lambda_j-\mu_j}{\lambda_k - \lambda_j},
$$
we have that $c_k \ne 0$. Remark that $c_k = \lim_{N \rightarrow \infty} c_{k,N} $. Let $L \subset \C \setminus F$ be a compact set. We have that
\begin{align*}
  \sup_{z \in L} \abs{f(z) - f_N(z)}
    &\le \sup_{z \in L} \abs{ \sum_{i=1}^N \frac{c_i - c_{i,N}}{z-\lambda_i} }
          + \sup_{z \in L} \abs{ \sum_{i=N+1}^\infty \frac{c_i}{z-\lambda_i} } \\
    &\le \frac{1}{\dist(L,F)}\left( \sum_{i=1}^N \abs{c_i - c_{i,N}} + \sum_{i=N+1}^\infty \abs{c_i} \right).
\end{align*}
As $(c_i)_{i \in \N} \in \lun $, we have that 
$$
  \lim_{N \rightarrow \infty} \sum_{i=N+1}^\infty \abs{c_i}  = 0.
$$
We have that
\begin{align*}
  \sum_{i=1}^N &\abs{c_i - c_{i,N}} \\
  &= \sum_{i=1}^N \abs{(\lambda_i - \mu_i) \prod_{j=1, \, j \ne i}^N \frac{\lambda_i-\mu_j}{\lambda_i - \lambda_j}} \abs{\prod_{j=N+1}^\infty\left( 1 + \frac{\lambda_j-\mu_j}{\lambda_i - \lambda_j}\right) - 1 } \\
  &= \sum_{i=1}^N \abs{\prod_{j=1}^{i-1} \frac{\lambda_i-\mu_j}{\lambda_i - \lambda_j}(\lambda_i - \mu_i)} \abs{\prod_{j=i+1}^N \frac{\lambda_i-\mu_j}{\lambda_i - \lambda_j}} \abs{\prod_{j=N+1}^\infty\left( 1 + \frac{\lambda_j-\mu_j}{\lambda_i - \lambda_j}\right) - 1 } \\
  &\le \sum_{i=1}^N \frac{\gamma_i}{\prod_{j=i+1}^\infty (1+\epsilon_j)} \prod_{j=i+1}^N (1+\epsilon_j) \abs{\prod_{j=N+1}^\infty(1+\epsilon_j) -1} \\
  &= \frac{\abs{\prod_{j=N+1}^\infty (1+\epsilon_j) - 1 }}{\prod_{j=N+1}^\infty (1+\epsilon_j)} \sum_{i =1}^N \gamma_i.
\end{align*}
  As $\sum_{i =1}^N \gamma_i \le \sum_{i =1}^\infty \gamma_i < \infty $ and as 
$$
  \lim_{N \rightarrow \infty} \frac{\abs{\prod_{j=N+1}^\infty (1+\epsilon_j) - 1 }}{\prod_{j=N+1}^\infty (1+\epsilon_j)} = 0,
$$
we obtain that $f_N$ converges to $f$  uniformly on every compact subset of $\C \setminus F$. As $f_N$ does not vanish on $\C \setminus F$, $f$ doesn't either.
\end{proof}

\section{Proof of the main Theorem}

We are now ready to prove Theorem \ref{ThPertRangUnOpDiag}.

\begin{proof}[Proof of Theorem \ref{ThPertRangUnOpDiag}]
  According to Proposition \ref{PropUPasDansImDMoinsZ}, there exists $u \in H$ such that for every $i \in \N$, $\ps{u}{e_i} \ne 0$ and for all $z \in \Spec{D} \setminus \Specp{D} $ we have that
  $$
    \sum_{i \in \N} \frac{\abs{\ps{u}{e_i}}^2}{\abs{z-\lambda_i}^2} = \infty.
  $$
  Let $\delta >0$. From Proposition \ref{PropFctionHoloQuiSannulePas}, is follows that there exist complex numbers $c_i$ such that for every $i \in \N$ we have
  $$
    0 < \abs{c_i} \le \delta \abs{\ps{u}{e_i}}^2,
  $$
  and for all $z \in \C \setminus \Spec{D}$, we have that
  $$
    \sum_{i \in \N} \frac{c_i}{z-\lambda_i} \ne 1.
  $$
  Let $v$ be a vector such that for every $i \in \N$
  $$
  \ps{v}{e_i} = \frac{\conj{c_i}}{\conj{\ps{u}{e_i}}}.
  $$  
  Then $v \in H$, because 
  $$
    \norme{v}^2 
      =   \sum_{i \in \N} \frac{\abs{c_i}^2}{\abs{\ps{u}{e_i}}^2} 
      \le \delta^2 \sum_{i \in \N} \abs{\ps{u}{e_i}}^2 
      = \delta^2 \norme{u}^2.
  $$
  This yields that for every $z \in \C \setminus \Spec{D}$, 
  $$
    \sum_{i \in \N} \frac{\ps{u}{e_i}\conj{\ps{v}{e_i}}}{z-\lambda_i} \ne 1.  
  $$
  From Proposition \ref{PropIonascuCnsRangUnVp}, we have that $\Specp{D + \tens{u}{v}} = \emptyset$. Indeed if $z \in \Specp{D}$ then $z \notin \Specp{D +\tens{u}{v}} $ because of condition (1). If $z \in \Spec{D} \setminus \Specp{D} $, then $ \sum_{i \in \N} \frac{\abs{\ps{u}{e_i}}^2}{\abs{z-\lambda_i}^2} = \infty $ , so $z \notin \Specp{D + \tens{u}{v}}$ because of condition (2). If $z \in \C \setminus \Spec{D}$, then $ \sum_{i \in \N} \frac{\ps{u}{e_i}\conj{\ps{v}{e_i}}}{z-\lambda_i} \ne 1 $ and $z \notin \Specp{D + \tens{u}{v}} $ because of condition (3). Moreover, we have that 
  $$
  \norme{\tens{u}{v}} \le \norme{u}\norme{v} \le \delta \norme{u}^2 .
  $$
  By choosing $\delta$ arbitrarily small, we can ensure that $\norme{\tens{u}{v}}$ is arbitrarily small.
  
  We still have to check that $\Spec{D} = \Spec{D + \tens{u}{v}}$. We have that $\Spece{D} = \Spec{D} $ because $\Spec{D}$ has no isolated points. So we get that $\Spec{D} = \Spece{D} \subset \Spec{D + \tens{u}{v}} $. If $z \in \Spec{D + \tens{u}{v} } \setminus \Spec{D}$, according to Weyl's Theorem, $z$ belongs to $\Specp{D+ \tens{u}{v}} $. But $D+ \tens{u}{v}$ has no eigenvalues. Therefore $\Spec{D} = \Spec{D + \tens{u}{v}}$.
\end{proof}

\begin{remark}
  Suppose that $D$ satisfies the hypothesis of Theorem \ref{ThPertRangUnOpDiag}. Moreover suppose that $\Spec{D}$ is connected and has a non empty interior. Then we don't know if the operator $T = D + \tens{u}{v}$ built before posses a non trivial invariant subspace. Indeed we can't use the fact that $T$ posses an eigenvalue, because $T$ was build without eigenvalue. We can't use Riesz-Dunfod functional calculus because $\Spec{T}= \Spec{D}$ is connected. As $u \notin \Range{(D-z)}$ for all $z \in \Spec{D}$, we can't use the techniques of \cite{Foias_Jung_Ko_Pearcy_2007}, \cite{Fang_Xia_2012} nor \cite{Klaja_2013}. In this case we can also prove that $u \in \lp{1+\varepsilon}(\lbrace e_k \rbrace) $ for every $\varepsilon > 0$ and $u \notin \lp{1}(\lbrace e_k \rbrace)$. Therefore the $\lp{1}$ condition of \cite{Fang_Xia_2012} for rank one perturbation of diagonal operators is sharp in some sense, and there is a few hope that the previous technique can be used for getting a complete solution the existence of non trivial hyperinvariant subspace of rank one perturbation of diagonal operators. 
  
  However we don't know if $T^*$ possess or not an eigenvalue.
\end{remark}

\section{The unbounded case}

In this section we will generalize Theorem \ref{ThPertRangUnOpDiag} to unbounded diagonal operators.

\begin{theorem}
  \label{ThPertRangUnOpDiagNonBorne}
  Let  $D = \sum_{i \in \N} \lambda_i \tens{e_i}{e_i} $ be a diagonal operator (possibly unbounded). Suppose that for every $i \ne j $, $\lambda_i \ne \lambda_j $. Moreover suppose that $\Spec{D} =  \adherence{\lbrace \lambda_i: i \in \N \rbrace} $ is a perfect set. Then there exist $u,v \in H$ such that $\Specp{D + \tens{u}{v}} = \emptyset$.  

Moreover we can choose $u,v \in H$ such that $\norme{\tens{u}{v}}$ is arbitrarily small.  
\end{theorem}

\begin{proof}
  In the previous proof, only Proposition \ref{PropUPasDansImDMoinsZ} does not work in the unbounded case. Recall that in the proof of Proposition \ref{PropUPasDansImDMoinsZ}, we reduced to the case $\Spec{D} \subset [0,1]\times [0,1]$, which is possible only if $ \Spec{D}$ is a bounded set. The strategy of the proof here is to write the diagonal operator $D$ as a direct sum of bounded diagonal operators.
  
  Denote by $\mathcal{C}_{n,k} = ]n,n+1]\times ]k,k+1]$. We have that $\C = \sqcup_{n,k \in \Z} \mathcal{C}_{n,k}$. For every $n,k \in \Z$ denote by $H_{n,k}$ the space
  $$
    H_{n,k} = \Span{e_i : \lambda_i \in \mathcal{C}_{n,k}}.
  $$
  We have that $H_{n,k}$ reduce $D$,
  $$
  D = \bigoplus_{n,k \in \Z} D_{|H_{n,k}}
  \quad \textrm{and} \quad
  H = \bigoplus_{n,k \in \Z} H_{n,k}.
  $$
  We want to apply Proposition \ref{PropUPasDansImDMoinsZ} to $ D_{|H_{n,k}}$, but this is not possible yet.
  
  If $z \in \Spec{ D_{|H_{n,k}}} \cap \Interieur{\mathcal{C}_{n,k}} $, then $z$ cannot be an isolated point in $\Spec{ D_{|H_{n,k}}}$, otherwise it would be also isolated in $ \Spec{D}$, this would contradicts the hypothesis that $\Spec{D}$ is a perfect set.
  
  If $z \in \Spec{ D_{|H_{n,k}}} \cap \mathcal{C}_{n,k} \setminus \Interieur{\mathcal{C}_{n,k}} $ is an isolated point of $\Spec{D_{|H_{n,k}}} $, then $ z \in \Specp{D}$ (otherwise, as $z \in \Spec{D_{|H_{n,k}}}$, $z$ would be the limit of some sequence $\lambda_i \in \mathcal{C}_{n,k} $ and wouldn't be isolated). Therefore there exists $i \in \N$ such that $z = \lambda_i$. As $z\in \Spec{D}$ and $\Spec{D} $ is a perfect set, $z= \lambda_i$ is not isolated in $\Spec{D_{|H_{n+1,k}}} $, $\Spec{D_{|H_{n,k+1}}}$ or $D_{|H_{n+1,k+1}}$.
  
  In order to avoid that  $z= \lambda_i$ is an isolated point of $\Spec{D_{|H_{n,k}}} $, we have to put the vector $e_i$ in the good subspace near $H_{n,k}$.
  
  Denote by $\tilde{H}_{n,k} $ the previous cutting of $H$ which take account of this last precaution. We still have that $\tilde{H}_{n,k}$ reduce $D$,
  $$
  D = \bigoplus_{n,k \in \Z} D_{|\tilde{H}_{n,k}},
  \quad \textrm{and} \quad
  H = \bigoplus_{n,k \in \Z} \tilde{H}_{n,k}.
  $$
  
  So $\Spec{D_{|\tilde{H}_{n,k}}} $ is a perfect compact set and we can apply Proposition \ref{PropUPasDansImDMoinsZ}.  There exists $u_{n,k} $ such that for every $z \in \Spec{D_{|\tilde{H}_{n,k}}} \setminus \Specp{D_{|\tilde{H}_{n,k}}}$, 
  $
  u_{n,k} \notin \Range(D_{|\tilde{H}_{n,k}}-z)
  $.
  So for all $\alpha_{n,k} > 0$, we have for every $z \in \Spec{D_{|\tilde{H}_{n,k}}} \setminus \Specp{D_{|\tilde{H}_{n,k}}}$ that 
  $
   \alpha_{n,k} u_{n,k} \notin \Range(D_{|\tilde{H}_{n,k}}-z)
  $.
  We choose a sequence of positive numbers $\alpha_{n,k}$ such that
  $$
    \sum_{n \in \Z} \sum_{k \in \Z} \alpha_{n,k}^2 \norme{u_{n,k}} < \infty.
  $$
 Denote by
 $$
   u = \bigoplus_{n,k \in \Z} \alpha_{n,k}u_{n,k}.  
 $$
we get that $u \in H$. Moreover we have that for every $z \in \Spec{D} \setminus \Specp{D}$, there exists $n,k \in \Z$ such that $z \in \Spec{D_{|\tilde{H}_{n,k}}}\setminus \Specp{D_{|\tilde{H}_{n,k}}} $.
  Indeed, if $z \in \Spec{D}$, then there exists a sequence  $\lambda_i$ of eigenvalues of $D$ converging to $z$. If we consider a good subsequence, we can assume that every $\lambda_i$ are in only one square $C_{m,j}$. From the construction of $\tilde{H}_{m,j}$, $\lambda_i$ is associated to an eigenvector $e_i$ whether in $\tilde{H}_{m,j}$, $\tilde{H}_{m+1,j}$, $\tilde{H}_{m,j+1}$ or $\tilde{H}_{m+1,j+1}$. From the drawer principle, there exist at least one of those subspace which contains infinitely many $e_i$. 
   Denote this subspace by $\tilde{H}_{n,k}$. Then there exist infinitely many $\lambda_i$ in the spectrum of $D$ restricted to $\tilde{H}_{n,k}$. As $\lambda_i$ converge to $z$ and the spectrum is closed, we get that $z \in \Spec{D_{\vert \tilde{H}_{n,k}}}$. 
   As
  $$
  u_{n,k} \notin \Range(D_{|\tilde{H}_{n,k}} -z) 
  $$
  and thus 
  $$
    u \notin \Range(D-z).
  $$
  
  As Proposition \ref{PropFctionHoloQuiSannulePas} does not require that $\Spec{D}$ is bounded, we can finish the proof as in Theorem \ref{ThPertRangUnOpDiag}.

\end{proof}

\section*{Acknowledgments}
    I would like to thank Sophie Grivaux for several discussions and for her help to improve this paper. I would like to thank also William Alexandre for the proof of Proposition \ref{PropFctionHoloQuiSannulePas}, and for kindly allowing me to reproduce it here. I would like also to thanks the referee for his suggestions that helped to improve the presentation of the paper.

\bibliographystyle{alpha}
\bibliography{Biblio}

\begin{thebibliography}{FJKP07}

\bibitem[Ale]{William_Alexandre}
William Alexandre.
\newblock Personal communication.

\bibitem[BL15]{Baranov_Lishanskii_2014}
Anton Baranov and Andrei Lishanskii.
\newblock On {S}. {G}rivaux' example of a hypercyclic rank one perturbation of
  a unitary operator.
\newblock {\em Arch. Math. (Basel)}, 104(3):223--235, 2015.

\bibitem[FJKP07]{Foias_Jung_Ko_Pearcy_2007}
Ciprian Foias, Il~Bong Jung, Eungil Ko, and Carl Pearcy.
\newblock On rank-one perturbations of normal operators.
\newblock {\em J. Funct. Anal.}, 253(2):628--646, 2007.

\bibitem[FX12]{Fang_Xia_2012}
Quanlei Fang and Jingbo Xia.
\newblock Invariant subspaces for certain finite-rank perturbations of diagonal
  operators.
\newblock {\em J. Funct. Anal.}, 263(5):1356--1377, 2012.

\bibitem[Gri12]{Sophie_2012}
Sophie Grivaux.
\newblock A hypercyclic rank one perturbation of a unitary operator.
\newblock {\em Math. Nachr.}, 285(5-6):533--544, 2012.

\bibitem[Ion01]{Ionascu_2001}
Eugen~J. Ionascu.
\newblock Rank-one perturbations of diagonal operators.
\newblock {\em Integral Equations Operator Theory}, 39(4):421--440, 2001.

\bibitem[Kla15]{Klaja_2013}
Hubert Klaja.
\newblock Hyperinvariant subspaces for some compact perturbations of
  multiplication operators.
\newblock {\em J. Operator Theory}, 73(1):127--142, 2015.

\bibitem[Rog70]{Rogers_1970}
C.~A. Rogers.
\newblock {\em Hausdorff measures}.
\newblock Cambridge University Press, London, 1970.

\bibitem[SS05]{Stein_Shakarchi_2005}
Elias~M. Stein and Rami Shakarchi.
\newblock {\em Real analysis}.
\newblock Princeton Lectures in Analysis, III. Princeton University Press,
  Princeton, NJ, 2005.
\newblock Measure theory, integration, and Hilbert spaces.

\bibitem[Sta84]{Stampfli_2001}
Joseph~G. Stampfli.
\newblock One-dimensional perturbations of operators.
\newblock {\em Pacific J. Math.}, 115(2):481--491, 1984.

\end{thebibliography}

\end{document}